\documentclass[12pt,reqno]{amsart}

\usepackage{amssymb}
\usepackage{amscd}
\usepackage{amsfonts}
\usepackage{setspace}
\usepackage{mathtools}
\usepackage{enumitem}
\usepackage{xcolor}
\usepackage[pdfencoding=auto]{hyperref}
\hypersetup{
  colorlinks,
  linkcolor={blue!50!black},
  citecolor={green!40!black},
  urlcolor={blue!80!black}
}
\usepackage[capitalize]{cleveref}

\usepackage{graphicx}

\newtheorem{theorem}{Theorem}[section]
\newtheorem{lemma}[theorem]{Lemma}
\newtheorem{proposition}[theorem]{Proposition}
\newtheorem{corollary}[theorem]{Corollary}

\theoremstyle{definition}

\newtheorem{definition}[theorem]{Definition}

\newtheorem{remark}[theorem]{Remark}

\renewcommand{\leq}{\leqslant}
\renewcommand{\geq}{\geqslant}

\newcommand\D{\operatorname{D}}

\def\F{\mathbf{F}}

\def\C{\mathbf{C}}
\def\Z{\mathbf{Z}}
\def\E{\mathbf{E}}
\def\P{\mathbf{P}}

\def\H{\mathbf{H}}
\def\I{\mathbf{I}}

\def\eps{\varepsilon}
\newcommand{\dist}[2]{{\operatorname{d}[#1;#2]}} 

\newcommand{\bigdist}[2]{{\operatorname{d}\bigl[#1;#2\bigr]}} 
\newcommand{\Bigdist}[2]{{\operatorname{d}\Bigl[#1;#2\Bigr]}} 
\newcommand{\multidist}[1]{\D[#1]}
\newcommand{\bigmultidist}[1]{\D\bigl[#1\bigr]}
\newcommand{\Bigmultidist}[1]{\D\Bigl[#1\Bigr]}
\newcommand{\ent}[1]{\H[#1]}
\newcommand{\bigent}[1]{\H\bigl[#1\bigr]}
\newcommand{\Bigent}[1]{\H\Bigl[#1\Bigr]}
\newcommand{\mutual}[1]{\I[#1]}
\newcommand{\bigmutual}[1]{\I\bigl[#1\bigr]}
\newcommand{\Bigmutual}[1]{\I\Bigl[#1\Bigr]}
\newcommand{\tsum}[0]{{\textstyle\sum}}
\def\cI{\mathcal{I}}

\newcommand\xsum{S}

\newcommand{\uppar}[1]{\textup{(}#1\textup{)}} 

\numberwithin{equation}{section}

\begin{document}

\title[Marton's conjecture in bounded torsion]{Marton's conjecture in abelian groups with bounded torsion}

\author{W.~T.~Gowers}
\thanks{}
\address{Coll\`ege de France \\
11, place Marcelin-Berthelot \\
75231 Paris 05 \\
France \\
and Centre for Mathematical Sciences \\
Wilberforce Road \\
Cambridge CB3 0WB \\
UK}
\email{wtg10@dpmms.cam.ac.uk}

\author{Ben Green}
\thanks{BG is supported by Simons Investigator Award 376201.}
\address{Mathematical Institute \\
Andrew Wiles Building \\
Radcliffe Observatory Quarter \\
Woodstock Rd \\
Oxford OX2 6QW \\
UK}
\email{ben.green@maths.ox.ac.uk}

\author{Freddie Manners}
\thanks{FM is supported by a Sloan Fellowship.}
\address{Department of Mathematics \\
University of California, San Diego (UCSD)\\
9500 Gilman Drive \# 0112 \\
La Jolla, CA  92093-0112 \\ 
USA}
\email{fmanners@ucsd.edu}

\author{Terence Tao}
\thanks{TT is supported by NSF grant DMS-1764034.  }
\address{Math Sciences Building \\
520 Portola Plaza \\
Box 951555 \\
Los Angeles, CA 90095 \\
USA}
\email{tao@math.ucla.edu}

\begin{abstract}
  We prove a Freiman--Ruzsa-type theorem with polynomial bounds in arbitrary abelian groups with bounded torsion, thereby proving (in full generality) a conjecture of Marton.
Specifically, let $G$ be an abelian group of torsion $m$ (meaning $mg=0$ for all $g \in G$) and suppose that $A$ is a non-empty subset of $G$ with $|A+A| \leq K|A|$.
Then $A$ can be covered by at most $(2K)^{O(m^3)}$ translates of a subgroup of $H \leq G$ of cardinality at most $|A|$.
The argument is a variant of that used in the case $G = \F_2^n$ in a recent paper of the authors.
\end{abstract}
\maketitle

\section{Introduction}

This paper is a companion to the recent work~\cite{ggmt} of the authors. Here, we adapt the techniques of that paper to prove a conjecture of Marton, also known as the polynomial Freiman--Ruzsa conjecture, in all groups with bounded torsion.  We refer to our previous paper~\cite{ggmt} for further discussion of this conjecture.

Throughout the paper, $m \geq 2$ will be an integer. We say that an abelian group $G = (G,+)$ has \emph{torsion $m$} if $mg = 0$ for all $g \in G$. Note that we do not require $m$ to be minimal with this property.

Here is our main result.

\begin{theorem}\label{mainthm}
Let $G$ be an abelian group with torsion $m$ for some $m \geq 2$. Suppose that $A \subseteq G$ is a finite non-empty set with $|A + A| \leq K|A|$. Then $A$ is covered by at most\footnote{See \cref{notation-sec} for our conventions on asymptotic notation and sumsets.} $(2K)^{O(m^3)}$ cosets of some subgroup $H \leq G$ of size at most $|A|$. Moreover, $H$ is contained in $\ell A - \ell A$ for some $\ell \ll (2 + m\log K)^{O(m^3 \log m)}$.
\end{theorem}

The $m=2$ case of this theorem (without the containment $H \subseteq \ell A - \ell A$, but with the bound $(2K)^{O(m^3)}$ replaced by the more explicit quantity $2K^{12}$) was\footnote{This bound was subsequently improved by Jyun-Jie Liao, first to $2K^{11}$ and then to $2K^9$; see \cite{jjl}.} established in our previous paper~\cite{ggmt}.  The requirement that $G$ be finite can be easily removed if $A$ is finite, since the torsion hypothesis then implies that the group generated by $A$ is also finite with torsion $m$.  The dependence of the exponents in this theorem on $m$ can probably be improved with additional effort.

In~\cite[Theorem 11.1]{sanders} it was shown under the hypotheses of \cref{mainthm} that $2A-2A$ contains a subgroup $H'$ whose size is at least $\exp(-C_m \log^4(2K)) |A|$, for some $C_m$ depending on $m$, where here and throughout the paper we write $\log^k(x)$ for $(\log x)^k$. By standard arguments this also implies that $|H'| \leq K^{O(1)} |A|$ and that $A$ can be covered by $\exp(C_m \log^4(2K))$ cosets of $H'$ (possibly after increasing $C_m$ slightly). A somewhat stronger bound in the direction of \cref{mainthm} was obtained by Konyagin (see~\cite{sanders2}), who proved a corresponding result with $\exp(C_{m,\eps} \log^{3+\eps}(2K))$ translates of $H'$ for any $\eps>0$; however, in Konyagin's result, unlike in the work of Sanders, one does not have $H' \subseteq 2A-2A$.

It is a well-known conjecture, the polynomial Bogolyubov conjecture, that one can find a subgroup $H \subseteq   2A-2A$ with size as large as $K^{-O_m(1)} |A|$. We have been unable to achieve this even in the $m=2$ case. However, our method naturally gives the weaker statement given in \cref{mainthm}, i.e., that we can find a subgroup $H$ of this size in $\ell A - \ell A$ for $\ell$ polylogarithmic in $K$ (rather than $\ell=2$ as in the conjecture).

The special case of \cref{mainthm} that is likely to be of most interest is when $G = \F_p^n$ for some finite field $\F_p$, a situation which has been widely studied in the literature. Here of course we can take $m=p$.  This special case of the result may be formulated in various equivalent ways, analogously to~\cite[Corollary 1.3]{ggmt}. For applications, the most useful one is likely to be a polynomially effective inverse theorem for the $U^3(\F_p^n)$-norm
\[ \|f\|_{U^3(\F_p^n)} \coloneqq \left( \E_{x,h_1,h_2,h_3 \in \F_p^n} \Delta_{h_1} \Delta_{h_2} \Delta_{h_3} f(x) \right)^{1/8},\]
where $\Delta_h f(x) \coloneqq f(x) \overline{f(x+h)}$ and we use the averaging notation $\E_{x \in A} \coloneqq \tfrac{1}{|A|} \sum_{x \in A}$.

\begin{corollary}\label{inverse-cor} Let $p$ be odd.
Let $f \colon \F_p^n \rightarrow \C$ be a $1$-bounded function, and suppose that $\Vert f \Vert_{U^3(\F_p^n)} \geq \eta$ for some $\eta$, $0 < \eta \leq \tfrac{1}{2}$. Then there is a quadratic polynomial $\phi \colon \F_p^n \rightarrow \F_p$ such that \[ |\E_{x \in \F_p^n} f(x) e_p(-\phi(x))| \gg \eta^{O(p^3)},\] where $e_p(x) \coloneqq e^{2\pi i x/p}$.
\end{corollary}

For a proof of this corollary (which is well-known to follow from a result such as \cref{mainthm}, but this deduction is not explicitly in the literature), see \cref{u3-app}. The $p=2$ version of this corollary was established in~\cite{ggmt}.

As in~\cite{ggmt}, entropy notions will play a key role in the proof of \cref{mainthm}. For any set $G$, define a $G$-valued random variable to be a discrete random variable taking finitely many values in $G$. The \emph{entropy} $\ent{X}$ of a $G$-valued random variable is defined by the formula
\[ \ent{X} \coloneqq \sum_{x \in G} p_X(x) \log \frac{1}{p_X(x)}\]
where $p_X(x) \coloneqq \P(X=x)$ is the distribution function of $X$, with $\log$ denoting the natural logarithm, and with the usual convention $0 \log \tfrac{1}{0} = 0$.  The \emph{entropic Ruzsa distance} $\dist{X}{Y}$ between two $G$-valued random variables $X, Y$, first introduced in~\cite{ruzsa-entropy} and then studied further in~\cite{ggmt,gmt,tao-entropy}, is defined by the formula
\begin{equation}\label{ruzsa-def}
 \dist{X}{Y} \coloneqq \ent{\tilde X + \tilde Y} - \tfrac{1}{2} \ent{\tilde X} - \tfrac{1}{2} \ent{\tilde Y},
\end{equation}
where $\tilde X, \tilde Y$ are independent copies of $X,Y$ respectively.

Some basic inequalities involving entropic Ruzsa distance (valid for arbitrary abelian groups $G$) are collected in \cref{entropy-app}. For now, we give three representative examples of these inequalites that will be used repeatedly in our arguments.
The first is the entropic Ruzsa triangle inequality
\[ \dist{X}{Y} \leq \dist{X}{Z} + \dist{Z}{Y}\]
for any $G$-valued random variables $X,Y,Z$; see \cref{ruzsa-calculus}\ref{rc-i}.  The second is a Kaimonovich--Vershik--Madiman type inequality
\[ \bigdist{X}{\tsum_{i=1}^n Y_i} \leq 2 \sum_{i=1}^n \dist{X}{Y_i}\]
for any $G$-valued random variables $X,Y_1,\dots,Y_n$ with $Y_1,\dots,Y_n$ independent; see \cref{ruzsa-calculus}\ref{rc-x}.  The third is the inequality
\begin{equation}\label{crude-bukh}
  \ent{X-aY} - \ent{X} \ll (1 + \log |a|) \dist{X}{Y}
\end{equation}
for any non-zero integer $a$ and independent $G$-valued random variables $X,Y$; see \cref{lem-lin} (ii).  We note that a cruder inequality with $O(|a|)$ in place of $O(\log |a|)$ on the right-hand side is rather easier to prove (see \cref{lem-lin} (i)), and leads to the same exponents in our main results. However,~\eqref{crude-bukh} may be of some independent interest.

Another Ruzsa distance inequality worth mentioning is the contraction property
\[ \dist{\pi(X)}{\pi(Y)} \leq \dist{X}{Y}, \]
valid whenever $X,Y$ are $G$-valued random variables and $\pi \colon G \to G'$ is a homomorphism.  While we do not directly use this inequality in this paper, a more precise version of this inequality, which we call the ``fibring inequality'', played a crucial role in our previous paper~\cite{ggmt}, and a variant of that inequality, which we call the ``multidistance chain rule'', will play a similarly crucial role in our current arguments; see \cref{multidist-chain-rule} below. 

In~\cite{gmt}, it was shown (over $\F_2$) that the polynomial Freiman--Ruzsa conjecture can be formulated using the language of entropy, and this is the form of the conjecture we will work with here.

\begin{theorem}\label{main-entropy} Suppose that $G$ is an abelian group of torsion $m$.  Suppose that $X, Y$ are $G$-valued random variables. Then there exists a subgroup $H \leq G$ such that \[ \dist{X}{U_H}, \dist{Y}{U_H} \ll m^3 \dist{X}{Y}.\] Moreover, if $X,Y$ take values in some symmetric set $S \subseteq G$ containing the origin, then $H$ can be taken to be contained in $\ell S$ for some $\ell \ll (2 + m\dist{X}{Y})^{O(m^3 \log m)}$.
\end{theorem}

The equivalence of this result and \cref{mainthm} can be obtained by making simple modifications to the arguments of~\cite[Section 8]{gmt}, so that they work over arbitrary torsion groups rather than $\F_2$. The important direction for us, namely that \cref{main-entropy} implies \cref{mainthm}, is also the easier one. We give the argument in full in \cref{comb-app}.

A variant of \cref{main-entropy} in the case $m=2$ was established in~\cite{ggmt}.  There, the idea was to work with minimizers $X,Y$ of a certain functional
\[ \tau[X;Y] \coloneqq \dist{X}{Y} + \tfrac{1}{9} \dist{X^0}{X} + \tfrac{1}{9} \dist{Y^0}{Y} \]
where $X^0,Y^0$ were the original random variables of interest.
If one took $X,Y,\tilde X, \tilde Y$ to be four independent copies of $X,Y,X,Y$ respectively, it was shown in~\cite{ggmt} that entropy inequalities (and in particular the aforementioned  ``fibring identity'') could be used to establish that any two of the random variables $X+Y$, $\tilde X+Y$, $X+\tilde X$ were nearly independent of each other relative to $X+Y+\tilde X+\tilde Y$.
In an ``endgame'' step, the above fact was then shown to be in conflict with the characteristic $2$ identity
\[
  (X+Y) + (\tilde X+Y) + (X + \tilde X) = 0,
\]
except when $X, Y$ were already translates of a uniform distribution on a subspace.

In this paper we adopt a slight variant of the approach.
We replace the $\tau$ functional by a ``multidistance'' functional
\[
  \multidist{ (X_i)_{i=1}^m } = \bigent{ \tsum_{i=1}^m X_i } - \frac{1}{m} \sum_{i=1}^m \ent{X_i}
\]
(where we take $X_1,\dots,X_m$ to be independent here for simplicity of notation).
Instead of working with minimizers, we work with approximate minimizers, for which the multidistance cannot be significantly improved without making significant changes to the $X_i$.
Considering a tuple of $m^2$ independent random variables $X_{i,j}$ with $i,j \in \Z/m\Z$, with each $X_{i,j}$ an independent copy of $X_{i}$, we will use a kind of iterated variant of the fibring identity, which we call the \emph{multidistance chain rule}, to show that the variables
\[
  \sum_{i,j} i X_{i,j},\qquad \sum_{i,j} j X_{i,j},\qquad \sum_{i,j} (-i-j) X_{i,j}
\]
are pairwise approximately independent relative to $\sum_{i,j} X_{i,j}$.
The ``endgame'' is then based on finding a conflict with the obvious identity
\[
  \sum_{i,j} i X_{i,j} + \sum_{i,j} j X_{i,j} + \sum_{i,j} (-i-j) X_{i,j} = 0.
\]
This turns out to be enough to conclude using similar tools to those wielded in~\cite{ggmt}, that is to say entropic sumset estimates (which we informally term ``entropic Ruzsa calculus'') and in particular the entropic Balog--Szemer\'edi-Gowers theorem (established in~\cite[Lemma 3.3]{tao-entropy}, although we use a minor variant detailed in~\cite{ggmt}).

\subsection{Notation and terminology}\label{notation-sec} %
The notations $X = O(Y)$, $X \ll Y$, or $Y \gg X$ all mean that $|X| \leq CY$ for an absolute constant $C$. Different instances of the notation may imply different constants $C$.
We use the usual sumset notation $A \pm B \coloneqq \{a \pm b: a \in A, b \in B\}$ and $-A \coloneqq \{-a: a \in A \}$, and let $\ell A = A + \cdots + A$ denote the sum of $\ell$ copies of $A$. A set $S$ is symmetric if $S=-S$.

We define the \emph{mutual information} $\mutual{X:Y}$ by the formula
\[  \mutual{X:Y} \coloneqq \ent{X} + \ent{Y} - \ent{X,Y}. \]
The \emph{conditional entropy} $\ent{X|Y}$ of $X$ relative to $Y$ is given by the formula
\[  \ent{X|Y} \coloneqq \sum_y p_Y(y) \ent{X\,|\,Y\mathop{=}y} \]
where $y$ ranges over the support of $p_Y$, and $(X\,|\,Y\mathop{=}y)$ denotes the random variable $X$ conditioned to $Y\mathop{=}y$; we observe the chain rule
\begin{equation}\label{chain-rule}
  \ent{X,Y} = \ent{X|Y} + \ent{Y}
\end{equation}
and conditional chain rule
\begin{equation}\label{chain-rule'}
  \ent{X,Y|Z} = \ent{X|Y,Z} + \ent{Y\,|\,Z}
\end{equation}
for arbitrary random variables $X,Y,Z$ defined on the same sample space; see for instance~\cite[Appendix A]{ggmt}. Here and in the sequel we are adopting the convention of omitting parentheses when the meaning is clear from context, for instance abbreviating $\ent{(X,Y)|(Z,W)}$ as $\ent{X,Y|Z,W}$, $\multidist{X_I|(Y_I,Z_I)}$ as $\multidist{X_I|Y_I,Z_I}$, and so forth.

The \emph{conditional mutual information} $\mutual{X:Y|Z}$ of two random variables $X,Y$ relative to a third $Z$ (that are all defined on the same sample space) is similarly defined by
\[  \mutual{X:Y|Z} \coloneqq \sum_z p_Z(z) \bigmutual{(X \,|\, Z\mathop{=}z) : (Y\,|\,Z\mathop{=}z)} \]
or equivalently
\begin{align}
  \nonumber
  \mutual{ X:Y|Z } &= \ent{X|Z} + \ent{Y|Z} - \ent{X,Y|Z}\\
  & = \ent{X,Z} + \ent{Y,Z} - \ent{X,Y,Z} - \ent{Z}.
  \label{cond-form-mutual}
\end{align}
We recall the submodularity inequality
\begin{equation}\label{nonneg-cond}
  \mutual{ X:Y \,|\,Z } \geq 0;
  \end{equation}
see e.g.,~\cite[(A.8)]{ggmt}.

In a similar vein, if $X$ and $(Y,W)$ are random variables, with $X, Y$ taking values in an abelian group $G$, we define the \emph{conditional Ruzsa distance} $\dist{X}{Y|W}$ by the formula
\begin{equation}\label{cond-form}
  \dist{X}{Y|W} = \sum_w p_W(w) \bigdist{X}{(Y\,|\,W\mathop{=}w)}.
\end{equation}
This is a special case of a somewhat more general conditional Ruzsa distance $\dist{X|Z}{Y|W}$ in which one conditions $X$ and $W$ separately; see~\cite{ggmt}.  However, we will only need the partially conditional Ruzsa distance~\eqref{cond-form} in this paper.\vspace*{8pt}

\emph{Acknowledgement.} We thank Zach Hunter for a number of comments and corrections on an earlier version of the paper.

\section{Induction on multidistance}

To describe our argument, we first introduce the notion of \emph{multidistance}. The need for this notion is a key innovation in this paper that was not required in~\cite{ggmt}.

It is convenient to introduce the following notational convention: if $I$ is a finite indexing set, then $X_I$ denotes a tuple $(X_i)_{i \in I}$ of random variables (and similarly, for example, $Y_I$ denotes $(Y_i)_{i \in I}$ and $X'_I$ denotes $(X'_i)_{i \in I}$). Usually all the variables in such a tuple will be $G$-valued for some abelian group $G$.

\begin{definition}\label{multi-def}
Let $G$ be an abelian group and let $X_{I}$ be a non-empty finite tuple of $G$-valued random variables. Then we define
\[
  \multidist{ X_{I}} \coloneqq \bigent{ \tsum_{i \in I} \tilde X_i } - \frac{1}{|I|} \sum_{i \in I} \ent{\tilde X_i},
\]
where the $\tilde X_i$ are independent copies of the $X_i$.
\end{definition}
From \cref{ruzsa-calculus}\ref{rc-ii} we see that $\bigent{\sum_{i \in I} \tilde X_i} \geq \ent{X_i}$ for all $i \in I$, and so on averaging we conclude that the multidistance is always non-negative.  It is also clearly invariant with respect to permutations of the $X_i$. We remark that, in the case $I = \{1,2\}$, $\multidist{ X_{\{1,2\}}}$ is equal to $\dist{X_1}{-X_2}$ (and hence, if $G$ is a vector space of characteristic $2$, is the same as $\dist{X_1}{X_2}$). This observation explains why we use the term multidistance; however, one should not take this terminology too seriously. The basic properties of multidistance are further developed in  \cref{ruzsa-multid-sec} below.

We will deduce \cref{main-entropy} from the following statement involving multidistance.

\begin{proposition}\label{thm-main-multi} Let $G$ be an abelian group of torsion $m$, and let $I = \{1,2,\dots, m\}$. If $X_{I}$ is a tuple of $G$-valued random variables then there exists a subspace $H\leq G$ such that
\[ \sum_{i \in I} \dist{X_i}{U_H}  \ll  m^3 \multidist{ X_{I} }.\]
Moreover, if all the $X_i$ take values in some symmetric set $S \subseteq G$ containing the origin, then $H$ can be taken to be contained in $\ell S$ for some $\ell \ll (2 + \multidist{X_I})^{O(m^3 \log m)}$.
\end{proposition}
The deduction of \cref{main-entropy} from this is fairly routine and is given in \cref{ruzsa-multid-sec}.

We will prove \cref{thm-main-multi} by a kind of induction on multidistance. The key technical result which drives this is the following proposition (cf.~\cite[Proposition 2.1]{ggmt}, as well as~\cite[Remarks 2.3,~2.4]{ggmt}).

\begin{proposition}\label{thm-main-multi-dec} Let $G$ be an abelian group with torsion $m$. Set $\eta \coloneqq c/m^3$ for a sufficiently small absolute constant $c>0$, and set $I := \{1,2,\dots, m\}$. If $X_{I}$ is a tuple of $G$-valued random variables with $\multidist{ X_{I}} > 0$, then there exists a tuple $X'_{I}$ of $G$-valued random variables such that
we have the multidistance decrement
\begin{equation}\label{multidist-dec}
 \multidist{ X'_{I} } \leq (1 - \eta) \multidist{ X_{I} } - \eta \sum_{i \in I} \dist{X_i}{X'_i}.
\end{equation}
Moreover, if all the $X_i$ take values in some symmetric set $S \subseteq G$ containing the origin, then the $X'_i$ can be chosen to take values in $m^3 S$.
\end{proposition}

The proof of this proposition occupies the bulk of the paper.

We will also need the following result concerning very small values of the multidistance, which forms the base case of the induction.

\begin{proposition}\label{multi-vsmall}
Let $G$ be a finite abelian group. Suppose that $I$ is an indexing set of size $m \geq 2$. Suppose that $X_{I}$ is a tuple of $G$-valued random variables with $\multidist{X_{I} } < c_0$ for a sufficiently small absolute constant $c_0 >0$. Then there is some subgroup $H \leq G$ such that $\sum_{i \in I} \dist{X_i}{U_H} \ll m \multidist{X_{I}} $.
Moreover, if all the $X_i$ take values in some symmetric set $S \subseteq G$ containing the origin, then we can take $H \subseteq 6S$.
\end{proposition}

This is a relatively straightforward consequence of~\cite[Theorem 1.3]{gmt}; for the details, see \cref{ruzsa-multid-sec}.

To conclude this outline, we show how our main result follows from \cref{thm-main-multi-dec} and \cref{multi-vsmall}.

\begin{proof}[Proof of \cref{thm-main-multi} assuming \cref{thm-main-multi-dec,multi-vsmall}]
  We apply \cref{thm-main-multi-dec} iteratively, obtaining for each $t \geq 0$ a tuple of random variables $X_{I}^{(t)}$ supported in $m^{3t} S$ with $X_I^{(0)} = X_I$ (that is, $X_i^{(0)} = X_i$ for $i \in I$), and
\begin{equation}\label{multidist-dec-t}
 \bigmultidist{ X^{(t+1)}_{I} } \leq (1 - \eta) \bigmultidist{ X^{(t)}_{I} } - \eta \sum_{i \in I} \bigdist{X^{(t)}_i}{X^{(t+1)}_i},
\end{equation}
which in particular implies that $\bigmultidist{ X^{(t+1)}_{I} } \leq (1 - \eta)  \bigmultidist{ X^{(t)}_{I} }$. Set $k := \multidist{X_I}$. By an easy induction it therefore follows that
\begin{equation}\label{dec-t}  \bigmultidist{ X^{(t)}_{I} } \leq (1 - \eta)^t k. \end{equation}
We apply this iteration until $t$ reaches the value
\begin{equation}\label{eta-s}
  s \coloneqq \lfloor C m^3  \log(2+k) \rfloor,
\end{equation}
where $C$ is a sufficiently large absolute constant.
By~\eqref{dec-t}, we will have $\bigmultidist{X^{(s)}_I }< c_0$ for $C$ large enough, where $c_0$ is the constant in \cref{multi-vsmall}.
From that proposition, we see that there is some subgroup $H \leq G$ such that $\sum_{i \in I} \bigdist{X^{(s)}_i}{U_H} \ll m \bigmultidist{X^{(s)}_I }$.
Moreover, we can take $H \subseteq \ell S$ for some $\ell \leq 6m^{3s} \ll (2+k)^{O(Cm^3 \log m)}$.

From several applications of the triangle inequality (that is, \cref{ruzsa-calculus}\ref{rc-i}) and~\eqref{multidist-dec-t} we now obtain
\begin{align*}
  \sum_{i \in I} \dist{X_i}{U_H}
  & \leq \sum_{i \in I} \bigdist{X_i^{(s)}}{U_H} + \sum_{t = 0}^{s-1}\sum_{i \in I} \bigdist{X_i^{(t)}}{X_i^{(t+1)}} \\
  & \ll m \multidist{X_I^{(s)}} +   \sum_{t = 0}^{s-1}\frac{1}{\eta} \Bigl(\bigmultidist{X_I^{(t+1)}} - (1 - \eta) \multidist{X_I^{(t)}} \Bigr) \\
  & \ll (m + \eta^{-1}) \bigmultidist{X_I^{(s)}} + \sum_{t = 1}^{s-1} \bigmultidist{X_I^{(t)}}.
\end{align*}
By~\eqref{dec-t} we deduce
\[
  \sum_{i \in I} \dist{X_i}{U_H}
\ll \big((m + \eta^{-1})(1 - \eta)^s + \eta^{-1} \big) k \ll m^3 k.
\]
This concludes the proof.
\end{proof}

\begin{remark}%
  \label{rem:different}
  The argument here is a little different from (but on some level equivalent to) that in~\cite{ggmt}, where a compactness argument was used. We hope that the reader will find it instructive to see the different arguments. The more hands-on iterative argument here allows one to retain some control of $H$ in terms of the support of the $X_i$, though at the cost of a slight degradation of the absolute constants (but not of the basic form of the dependence on $m$).
\end{remark}

The tasks for the remainder of the paper, then, are to prove \cref{thm-main-multi-dec} and \cref{multi-vsmall}, and to deduce \cref{mainthm} from \cref{thm-main-multi}. The second and third of these tasks are relatively straightforward and will be dealt with in the next section, leaving only the proof of \cref{thm-main-multi-dec} for the remaining sections of the paper.

\section{Relating Ruzsa distance and multidistance}\label{ruzsa-multid-sec}

We develop some basic properties of multidistance (\cref{multi-def}), starting by relating this notion to the more standard notion of entropic Ruzsa distance.

\begin{lemma}%
  \label{lem:ruzsa-multi}
  Let $G$ be an abelian group, let $I$ be an indexing set of size $m \ge 2$, and let $X_{I}$ be a tuple of $G$-valued random variables. Then
  \begin{enumerate}[label=\textup{(\roman*)}]
    \item $\sum_{\substack{j,k \in I \\ j \neq k}} \dist{X_j}{-X_k} \leq m(m-1) \multidist{X_I}$;
    \item $\sum_{j \in I} \dist{X_j}{X_j} \leq 2 m \multidist{X_I}$;
    \item if $(X_i)_{i \in I}$ all have the same distribution, $\multidist{X_I} \leq m \dist{X_j}{X_j}$ for any $j \in I$.
  \end{enumerate}
\end{lemma}
\begin{proof}%
Without loss of generality we may take the $X_i$ to be jointly independent.
From \cref{ruzsa-calculus}\ref{rc-ii}, we see that for any distinct $j,k \in I$, we have
\[
  \ent{X_j+X_k} \leq \bigent{ \tsum_{i \in I} X_i } ,
\]
and hence by~\eqref{ruzsa-def} we have
\[
  \dist{X_j}{-X_k} \leq \bigent{\tsum_{i \in I} X_i} - \tfrac{1}{2} \ent{X_j} - \tfrac{1}{2} \ent{X_k} .
\]
Summing this over all pairs $(j,k)$, $j \neq k$, gives (i).

From the triangle inequality (\cref{ruzsa-calculus}\ref{rc-i}) we have
\[  \dist{X_j}{X_j} \leq 2 \dist{X_j}{-X_k},  \]
and applying this to every summand in (i) gives (ii) (after dividing by $m-1$).

For (iii), we apply \cref{ruzsa-calculus}\ref{rc-x} with $X$ a further independent copy of $-X_j$, and $Y_1,\dots,Y_n$ a relabeling of $(X_i)_{i \in I}$, to get
\[
  \bigent{-X + \tsum_{i \in I} X_i} - \ent{X} \leq m \dist{X}{X}.
\]
Since
\[
  \bigent{-X + \tsum_{i \in I} X_i} \geq \bigent{\tsum_{i \in I} X_i}
\]
by \cref{ruzsa-calculus}\ref{rc-ii}, this implies (iii).
\end{proof}

We are now in a position to give the reduction of \cref{main-entropy} to  \cref{thm-main-multi}.

\begin{proof}[Proof of \cref{main-entropy} assuming \cref{thm-main-multi}]
  First, we claim that it suffices to establish \cref{main-entropy} in the case $X=Y$ (at the expense of worsening the implicit constants by a little more than a factor of 2).
  Indeed, for general $X,Y$, by the triangle inequality we have $\dist{X}{X} \leq 2 \dist{X}{Y}$, and so (assuming the case $X = Y$ of \cref{main-entropy}) there is some $H$ with
  \[ \dist{X}{U_H} \ll m^3 \dist{X}{Y}.\]
  We then also have
  \[ \dist{Y}{U_H} \leq \dist{X}{U_H} + \dist{X}{Y} \ll m^3 \dist{X}{Y}.\]
  This proves the claim.

  Suppose then that $X = Y$.
  Take $I = \{1,2,\dots, m\}$, and let $X_i = X$ for all $i \in I$.  From \cref{lem:ruzsa-multi}(iii) we have
  \[
    \multidist{ X_{I} } \leq m \dist{X}{X},
  \]
  and hence by \cref{thm-main-multi} we can find a subspace $H$ of $G$ such that $H \subseteq \ell S$ for some
  \[
    \ell \ll (2 + \multidist{ X_I })^{O(m^3 \log m)} \leq (2 + m\dist{X}{X})^{O(m^3 \log m)}
  \]
  and such that
  \[
    \sum_{i \in I} \dist{X}{U_H} \ll m^3 \multidist{X_I} \leq m^4 \dist{X}{X}.
  \]
  Since the left-hand side is $m \dist{X}{U_H}$, the theorem follows.
\end{proof}

We are also in a position to establish \cref{multi-vsmall}.

\begin{proof}[Proof of \cref{multi-vsmall}]
  Write $\eps := \multidist{X_I}$. Then by \cref{lem:ruzsa-multi}(i) and averaging, we can find $i \in I$ such that
\begin{equation}\label{kim}
  \sum_{k \neq i} \dist{X_i}{-X_k} \leq (m-1)\eps,
\end{equation}
and by further averaging we can find $j \neq i$ such that
\[  \dist{X_i}{-X_j} \leq \eps. \]
For $\eps$ small enough, we may apply~\cite[Theorem 1.3]{gmt} to conclude that there exists a finite subgroup $H$ of $G$ with $\dist{X_i}{U_{H}} \leq 12 \eps$.
An inspection of the proof of that theorem (in~\cite[\S 5]{gmt}) also reveals that $H$ is of the form $H = S'-S'$, where all the elements $y$ of $S'$ have Kullback--Leibler divergences\footnote{We refer the reader to~\cite{gmt} for a definition of Kullback--Leibler divergence.} $D_{\mathrm{KL}}(y-X_j||X_i-X_j)$ finite; this implies that $S' \subseteq 3S$ and hence that $H \subseteq 6S$.
Since $\dist{-Y}{U_H} = \dist{Y}{-U_H} = \dist{Y}{U_H}$ for any random variable $Y$,~\eqref{kim} and the triangle inequality (\cref{ruzsa-calculus}\ref{rc-i}) then give us
\[
  \sum_{k \in I} \dist{X_k}{U_H} \leq (13m - 1)\eps,
\]
as required.
\end{proof}

\section{The multidistance chain rule}\label{chain-sec}

In this section we establish a key inequality for the behaviour of multidistance under homomorphisms, together with some consequences of it. The key lemma, \cref{multidist-chain-rule}, is a ``chain rule'' for multidistance, analogous to the chain rule~\eqref{chain-rule} for Shannon entropy, as well as the ``fibring lemma'' for entropic Ruzsa distance in~\cite[Proposition 4.1]{ggmt} and~\cite[Proposition 1.4]{gmt}.

We have already remarked on our convention of writing $X_I = (X_i)_{i \in I}$ for a tuple of random variables indexed by some finite set $I$. If these random variables are $G$-valued, it is convenient to introduce two further notational conventions.  First, if $\pi \colon G \rightarrow H$ is a homomorphism, we write $\pi(X_I) \coloneqq (\pi(X_i))_{i \in I}$. Second, if $Y_I = (Y_i)_{i \in I}$ is another tuple of $G$-valued random variables, we write $X_I + Y_I \coloneqq (X_i + Y_i)_{i \in I}$.

 We will also need to introduce a notion of \emph{conditional multidistance}.  If $X_I$ and $Y_I$ are tuples of random variables, with the $X_i$ being $G$-valued, then we define
\begin{equation}\label{multi-def-cond}
\multidist{ X_{I} | Y_{I} }  \coloneqq \Bigent{ \tsum_{i \in I} \tilde X_i \big| (\tilde Y_j)_{j \in I} } - \frac{1}{|I|} \sum_{i \in I} \ent{\tilde X_i | \tilde Y_i}
  \end{equation}
where $(\tilde X_i,\tilde Y_i)$, $i \in I$ are independent copies of $(X_i,Y_i), i \in I$ (but note here that we do \emph{not} assume $X_i$ are independent of $Y_i$, or $\tilde X_i$ independent of $\tilde Y_i$).
Equivalently, one has
\begin{equation}\label{multi-def-cond-alt}
  \multidist{ X_{I} | Y_{I} } = \sum_{(y_i)_{i \in I}} \biggl(\prod_{i \in I} p_{Y_i}(y_i)\biggr) \bigmultidist{ (X_i \,|\, Y_i \mathop{=}y_i)_{i \in I} }
\end{equation}
where each $y_i$ ranges over the support of $p_{Y_i}$ for $i \in I$.

Here is the first key result of this section, the chain rule for multidistance.

\begin{lemma}\label{multidist-chain-rule}  Let $\pi \colon G \to H$ be a homomorphism of abelian groups and let $X_{I}$ be a tuple of jointly independent $G$-valued random variables.  Then $\multidist{ X_{I} }$ is equal to
\begin{equation}
    \bigmultidist{ X_{I} | \pi(X_{I}) }  + \bigmultidist{ \pi(X_{I}) }  + \Bigmutual{ \tsum_{i \in I} X_i  : \pi(X_{I}) \; \big| \; \pi\bigl(\tsum_{i \in I} X_i\bigr) }.
  \label{chain-eq}
\end{equation}
\end{lemma}

\begin{proof} For notational brevity during this proof, write $\xsum:= \sum_{i \in I} X_i$.

  Expanding out the definition~\eqref{cond-form-mutual} of $\mutual{\xsum:\pi(X_I)|\pi(\xsum)}$ and using the fact that $\pi(\xsum)$ is determined both by $\xsum$ and by $\pi(X_I)$, we obtain
\begin{equation*}
 \mutual{\xsum:\pi(X_I)|\pi(\xsum)} = \ent{\xsum}+\ent{\pi(X_I)}-\ent{\xsum,\pi(X_I)}-\ent{\pi(\xsum)},
\end{equation*}
and by the chain rule~\eqref{chain-rule} the right-hand side is equal to
\begin{equation*}
\ent{\xsum}-\ent{\xsum|\pi(X_I)}-\ent{\pi(\xsum)}.
\end{equation*}
Therefore,
\begin{equation}\label{chain-1}
\ent{\xsum}=\ent{\xsum|\pi(X_I)}+\ent{\pi(\xsum)}+\mutual{\xsum:\pi(X_I)|\pi(\xsum)}. \end{equation}
From a further application of the chain rule~\eqref{chain-rule} we have
\begin{equation}\label{chain-2}
  \ent{X_i} = \ent{X_i \, | \, \pi(X_i) } + \ent{\pi(X_i)}
\end{equation}
for all $i \in I$.  Averaging~\eqref{chain-2} in $i$ and subtracting this from~\eqref{chain-1}, we obtain \cref{multidist-chain-rule} as a consequence of the definition of multidistance (\cref{multi-def}).
\end{proof}

We will need to iterate the multidistance chain rule, so it is convenient to observe a conditional version of this rule, as follows.

\begin{lemma}\label{multidist-chain-rule-cond}
  Let $\pi \colon G \to H$ be a homomorphism of abelian groups.
  Let $I$ be a finite index set and let $X_{I}$ be a tuple of $G$-valued random variables.
  Let $Y_{I}$ be another tuple of random variables \uppar{not necessarily $G$-valued}.
  Suppose that the pairs $(X_i, Y_i)$ are jointly independent of one another \uppar{but $X_i$ need not be independent of $Y_i$}.
  Then
  \begin{align}\nonumber
      \multidist{ X_{I} | Y_{I} } &=  \bigmultidist{ X_{I} \,|\, \pi(X_{I}), Y_{I}} + \bigmultidist{  \pi(X_{I}) \,|\, Y_{I}} \\
       &\quad\qquad + \Bigmutual{ \tsum_{i\in I} X_i : \pi(X_{I}) \; \big| \;  \pi\bigl(\tsum_{i\in I} X_i \bigr), Y_{I} }.\label{chain-eq-cond}
  \end{align}
\end{lemma}

Indeed, for each $y_i$ in the support of $p_{Y_i}$, we may apply \cref{multidist-chain-rule} with $X_i$ replaced by the conditioned random variable $(X_i|Y_i=y_i)$, and the claim~\eqref{chain-eq-cond} follows by averaging~\eqref{chain-eq} in the $y_i$ using the weights $p_{Y_i}$.

We can iterate the above lemma as follows.

\begin{lemma}\label{multidist-chain-rule-iter}  Let $m$ be a positive integer.
  Suppose one has a sequence
  \begin{equation}\label{g-seq}
    G_m \to G_{m-1} \to \dots \to G_1 \to G_0 = \{0\}
  \end{equation}
  of homomorphisms between abelian groups $G_0,\dots,G_m$, and for each $d=0,\dots,m$, let $\pi_d \colon G_m \to G_d$ be the homomorphism from $G_m$ to $G_d$ arising from this sequence by composition \uppar{so for instance $\pi_m$ is the identity homomorphism and $\pi_0$ is the zero homomorphism}.
  Let $I$ be a finite index set and let $X_{I} = (X_i)_{i \in I}$ be a jointly independent tuple of $G_m$-valued random variables.
  Then
  \begin{equation}
    \begin{split}
      \multidist{ X_{I} } &=  \sum_{d=1}^m \bigmultidist{ \pi_d(X_{I}) \,|\, \pi_{d-1}(X_{I})} \\
       &\quad + \sum_{d=1}^{m-1} \Bigmutual{ \tsum_i X_i : \pi_d(X_{I}) \; \big| \; \pi_d\big(\tsum_i X_i\big), \pi_{d-1}(X_{I}) }.
    \end{split}\label{chain-eq-cond'}
  \end{equation}
  In particular, since all the $\mutual{-}$ terms are nonnegative \uppar{see~\eqref{nonneg-cond}}, we have
  \begin{align}\nonumber
      \multidist{ X_{I} } \geq  & \sum_{d=1}^m \bigmultidist{ \pi_d(X_{I})|\pi_{d-1}(X_{I}) } \\
       & + \Bigmutual{ \tsum_i X_i : \pi_1(X_{I}) \; \big| \; \pi_1\bigl(\tsum_i X_i\bigr) }.\label{chain-eq-cond''}
  \end{align}
\end{lemma}
\begin{proof}
Indeed, from \cref{multidist-chain-rule-cond} (taking $Y_I = \pi_{d-1}(X_I)$ and $\pi = \pi_d$ there, and noting that $\pi_d(X_I)$ determines $Y_I$) we have
\begin{align*}
  \bigmultidist{ X_{I} \,|\, \pi_{d-1}(X_{I}) } &=  \bigmultidist{ X_{I} \,|\, \pi_d(X_{I}) } + \bigmultidist{  \pi_d(X_{I})\,|\,\pi_{d-1}(X_{I}) } \\
                                           &\quad + \Bigmutual{ \tsum_{i \in I} X_i : \pi_d(X_{I}) \; \big| \; \pi_d\bigl(\tsum_{i \in I} X_i\bigr), \pi_{d-1}(X_{I}) }
\end{align*}
for $d=1,\dots,m-1$. The claim follows by telescoping series, noting that $\multidist{X_I | \pi_0(X_I)} = \multidist{X_I}$ and that $\pi_m(X_I)=X_I$.
\end{proof}

In our application we will need the following special case of the above lemma.

\begin{corollary}\label{cor-multid} Let $G$ be an abelian group and let $m \geq 2$.  Suppose that $X_{i,j}$, $1 \leq i, j \leq m$, are independent $G$-valued random variables.
  Then
  \begin{align*}
    &\Bigmutual{ \bigl(\tsum_{i=1}^m X_{i,j}\bigr)_{j =1}^{m} : \bigl(\tsum_{j=1}^m X_{i,j}\bigr)_{i = 1}^m \; \big| \; \tsum_{i=1}^m \tsum_{j = 1}^m  X_{i,j} } \\
    &\quad \leq \sum_{j=1}^{m-1} \Bigl(\bigmultidist{ (X_{i, j})_{i = 1}^m} - \Bigmultidist{ (X_{i, j})_{i = 1}^m  \; \big| \; (X_{i,j} + \cdots + X_{i,m})_{i =1}^m  }\Bigr) \\ & \qquad\qquad\qquad\qquad +  \bigmultidist{ (X_{i,m})_{i=1}^m} - \bigmultidist{ \bigl(\tsum_{j=1}^m X_{i,j}\bigr)_{i=1}^m },
  \end{align*}
where all the multidistances here involve the indexing set $\{1,\dots, m\}$.
\end{corollary}

\begin{proof}
  In \cref{multidist-chain-rule-iter} we take $G_d \coloneqq G^d$ with the maps $\pi_d \colon G^m \to G^d$ for $d=1,\dots,m$ defined by
\[
  \pi_d(x_1,\dots,x_m) \coloneqq (x_1,\dots,x_{d-1}, x_d + \cdots + x_m)
\]
with $\pi_0=0$. Since $\pi_{d-1}(x)$ can be obtained from $\pi_{d}(x)$ by applying a homomorphism, we obtain a sequence of the form~\eqref{g-seq}.

Now we apply \cref{multidist-chain-rule-iter} with $I = \{1,\dots, m\}$ and $X_i \coloneqq (X_{i,j})_{j = 1}^m$.  Using joint independence, we find that
\[
  \multidist{ X_{I} } = \sum_{j=1}^m \bigmultidist{ (X_{i,j})_{i \in I} }.
\]
On the other hand, for $1 \leq j \leq m-1$, we see that once $\pi_{j}(X_i)$ is fixed, $\pi_{j+1}(X_i)$ is determined by $X_{i, j}$ and vice versa, so
\[
  \bigmultidist{ \pi_{j+1}(X_{I}) \; | \; \pi_{j}(X_{I})} = \bigmultidist{ (X_{i, j})_{i \in I} \; | \; \pi_{j}(X_{I} )}.
\]
Since the $X_{i,j}$ are jointly independent, we may further simplify:
\[
  \bigmultidist{ (X_{i, j})_{i \in I} \; | \; \pi_{j}(X_{I})} = \bigmultidist{ (X_{i,j})_{i \in I} \; | \; ( X_{i, j} + \cdots + X_{i, m})_{i \in I} } .
\]
Putting all this into the conclusion of \cref{multidist-chain-rule-iter}, we obtain
\[
  \sum_{j=1}^{m} \multidist{ (X_{i,j})_{i \in I} }
  \geq
  \begin{aligned}[t]
  &\sum_{j=1}^{m-1} \bigmultidist{ (X_{i,j})_{i \in I} \; | \; (X_{i,j} + \cdots + X_{i,m})_{i \in I} } \\
  &\!\!\!+
  \Bigmultidist{ \bigl(\tsum_{j=1}^m X_{i,j}\bigr)_{i \in I}} \\
  &\!\!\!+\Bigmutual{  \bigl(\tsum_{i=1}^m X_{i,j}\bigr)_{j =1}^{m} : \bigl(\tsum_{j=1}^m X_{i,j}\bigr)_{i = 1}^m \; \big| \; \tsum_{i=1}^m \tsum_{j = 1}^m  X_{i,j} }
  \end{aligned}
\]
and the claim follows by rearranging. \end{proof}

\section{The main argument}

We now begin a preliminary discussion of the main task for the rest of the paper, the proof of \cref{thm-main-multi-dec}. As in~\cite{ggmt}, it is convenient to work in the contrapositive, which allows us to take advantage of the notation of conditional entropy.

Suppose, for the rest of the paper, that we have some $G$-valued random variables $X_i$, $i \in I = \{1,\dots, m\}$, taking values in a set $S \subseteq G$, with
\begin{equation}\label{k-def}
  k \coloneqq \multidist{X_{I}},
\end{equation}
and suppose that they cannot be decremented in the manner stated in \cref{thm-main-multi-dec}. That is to say,

\begin{equation}\label{main-contra}
 \multidist{ X'_{I} } \geq (1 - \eta) k - \eta \sum_{i \in I} \dist{X_i}{X'_i}
\end{equation}
for every tuple $X'_{I}$ of $G$-valued random variables taking values in $m^3 S$. The aim is then to show that $k = 0$; this is equivalent to \cref{thm-main-multi-dec}.

We now observe that~\eqref{main-contra} implies a conditioned variant of itself, namely
\begin{equation}\label{main-contra-cond}
 \multidist{ X'_{I} | Y_{I} } \geq (1 - \eta) k - \eta \sum_{i \in I} \dist{X_i}{X'_i|Y_i}
\end{equation}
for any tuple $X'_{I}$ of $G$-valued random variables taking values in $m^3 S$ and for any tuple $Y_{I}$ of random variables. Here, the conditioned multidistance $\multidist{ X'_{I} | Y_{I} }$ is defined as in~\eqref{multi-def-cond},~\eqref{multi-def-cond-alt}, and the conditional Ruzsa distance $\dist{X}{Y|Z}$ is defined in~\eqref{cond-form}. To obtain~\eqref{main-contra-cond} from~\eqref{main-contra}, simply replace $X'_i$ by $(X'_i | Y_i = y_i)$ and then sum weighted by $\prod_{i \in I} p_{Y_i}(y_i)$.

The inequality~\eqref{main-contra-cond} may be rearranged in the following convenient way:
\[ k - \multidist{ X'_{I} | Y_{I} } \leq \eta \biggl( k + \sum_{i \in I} \dist{X_i}{X'_i|Y_i}\biggr).\]
It also turns out to be convenient to note that
\begin{equation}\label{5.3-conv} k - \multidist{ X'_{I} | Y_{I} } \leq \eta \biggl( k + \sum_{i \in I} \dist{X_{\sigma(i)}}{X'_i|Y_i}\biggr)\end{equation} for any permutation $\sigma : I \rightarrow I$, since the multidistance is permutation invariant.

\subsection{Bounding the mutual information}

The first main step of the argument is to observe that~\eqref{5.3-conv} combines with \cref{cor-multid} to give the following inequality.

\begin{proposition}\label{key}
  Let $G$ be an abelian group.
  Let $m \geq 2$, and suppose that $X_{i,j}$, $1 \leq i,j \leq m$, are jointly independent $G$-valued random variables, such that for each $j = 1,\dots,m$, the random variables $(X_{i,j})_{i = 1}^m$ coincide in distribution with some permutation of the random variables $X_I = (X_i)_{i =1}^m$.
  Write
  \[
    \cI := \Bigmutual{ \bigl(\tsum_{i=1}^m X_{i,j}\bigr)_{j =1}^{m} : \bigl(\tsum_{j=1}^m X_{i,j}\bigr)_{i = 1}^m \; \big| \; \tsum_{i=1}^m \tsum_{j = 1}^m  X_{i,j} }.
  \]
  Then \uppar{assuming~\eqref{k-def} and~\eqref{main-contra} hold} we have
  \begin{equation}\label{I-ineq}
    \cI \leq 2\eta m \biggl( k +  \sum_{i=1}^m \dist{X_i}{X_i} \biggr) \leq 2 m (2m+1) \eta k.
  \end{equation}
\end{proposition}
For each $j \in \{1,\dots,m\}$ we call the tuple $(X_{i,j})_{i = 1}^m$ a \emph{column} and for each $i \in \{1,\dots, m\}$ we call the tuple $(X_{i,j})_{j = 1}^m$ a \emph{row}. Hence, by hypothesis, each column is a permutation of $X_I = (X_i)_{i=1}^m$.
\begin{proof}
  \cref{cor-multid} states that \begin{equation}\label{441} \cI \leq \sum_{j=1}^{m-1} A_j + B,\end{equation}
where
\[
  A_j \coloneqq \multidist{ (X_{i, j})_{i = 1}^m} - \bigmultidist{ (X_{i, j})_{i = 1}^m  \; \big| \; (X_{i,j} + \cdots + X_{i,m})_{i =1}^m  }
\]
and
\[
  B \coloneqq \bigmultidist{ (X_{i,m})_{i=1}^m} - \Bigmultidist{ \bigl(\tsum_{j=1}^m X_{i,j}\bigr)_{i=1}^m }.
\]
We first consider the $A_j$, for fixed $j \in \{1,\dots, m-1\}$.
By permutation symmetry of the multidistance, and our hypothesis on columns, we have
\[
  \multidist{ (X_{i, j})_{i = 1}^m} = \multidist{(X_i)_{i=1}^m} = \multidist{X_I} = k.
\]
Let $\sigma = \sigma_j \colon I \to I$ be a permutation such that $X_{i,j} = X_{\sigma(i)}$, and write $X'_i \coloneqq X_{i,j}$ and $Y_i \coloneqq X_{i,j} + \cdots + X_{i,m}$.
Note that all these random variables take values in $mS$.
By~\eqref{5.3-conv}, we conclude that
\begin{align}
  A_j & \leq \eta \biggl( k + \sum_{i = 1}^m \bigdist{X_{i,j}}{X_{i, j}|X_{i, j} + \cdots + X_{i,m}}\biggr).\label{54a}
\end{align}
We similarly consider $B$.  By permutation symmetry on the $m$-th column,
\[
  \bigmultidist{ (X_{i, m})_{i = 1}^m} = \multidist{X_I} = k.
\]
For $i \in I$, denote the sum of row $i$ by
\[
  V_i \coloneqq \sum_{j=1}^m X_{i,j};
\]
if we apply~\eqref{5.3-conv} again, now with $X_{\sigma(i)} = X_{i,m}$, $X'_i := V_i$, and with the variable $Y_i$ being trivial (that is, using the unconditioned statement~\eqref{main-contra}), we obtain
\begin{equation}\label{55a}
  B \leq \eta \biggl( k + \sum_{i = 1}^m \dist{X_{i,m}}{V_i}  \biggr).
\end{equation}

It remains to bound the distances appearing in~\eqref{54a} and~\eqref{55a} further using Ruzsa calculus.
For $1 \leq j \leq m-1$ and $1 \leq i \leq m$, by \cref{ruzsa-calculus}\ref{rc-ix} we have
\begin{align*} &\bigdist{X_{i,j}}{X_{i,j}| X_{i,j}+\cdots+X_{i,m}}
\leq \dist{X_{i,j}}{X_{i,j}} \\
&\quad + \tfrac{1}{2} \bigl(\ent{X_{i,j}+\cdots+X_{i,m}} - \ent{X_{i,{j+1}}+\cdots+X_{i,m}}\bigr).
\end{align*}
For each $i$, summing over $j = 1,\dots, m-1$ gives
\begin{align}
  \nonumber
  &\sum_{j=1}^{m-1} \bigdist{X_{i,j}}{X_{i,j}| X_{i,j}+\cdots+X_{i,m}} \\
  &\qquad \leq \sum_{j=1}^{m-1} \dist{X_{i,j}}{X_{i,j}} + \frac12 \bigl( \ent{V_i} - \ent{X_{i,m}} \bigr).
  \label{eq:distbnd1}
\end{align}
On the other hand, by \cref{lem:cool-dist-fact}(i) (since $X_{i,m}$ appears in the sum $V_i$) we have
\begin{align}
  \dist{X_{i,m}}{V_i}
  &\leq \dist{X_{i,m}}{X_{i,m}} + \frac12 \bigl( \ent{V_i} - \ent{X_{i,m}} \bigr).
  \label{eq:distbnd2}
\end{align}
Combining~\eqref{441},~\eqref{54a} and~\eqref{55a} with~\eqref{eq:distbnd1} and~\eqref{eq:distbnd2} (the latter two summed over $i$), we get
\begin{align}
  \nonumber
  \frac1{\eta} \cI &\leq m k + \sum_{i,j=1}^m \dist{X_{i,j}}{X_{i,j}} + \sum_{i=1}^m (\ent{V_i} - \ent{X_{i,m}}) \\
      &= m k + m \sum_{i=1}^m \dist{X_i}{X_i} + \sum_{i=1}^m \ent{V_i} - \sum_{i=1}^m \ent{X_i}.
      \label{eq:distbnd3}
\end{align}
By \cref{lem:cool-dist-fact}(ii) (with $f$ taking each $j$ to the index $j'$ such that $X_{i,j}$ is a copy of $X_{j'}$) we obtain the bound
\[
  \ent{V_i} \leq \Bigent{\tsum_{j=1}^m X_j} + \sum_{j=1}^m \dist{X_{i,j}}{X_{i,j}}.
\]
Finally, summing over $i$ and using $\multidist{X_I} = k$ gives
\begin{align*}
  \sum_{i=1}^m \ent{V_i} - \sum_{i=1}^m \ent{X_i} & \leq \sum_{i,j=1}^m \dist{X_{i,j}}{X_{i,j}} + m k \\ & = m\sum_{i = 1}^m \dist{X_i}{X_i} + mk,
\end{align*}
where in the second step we used the permutation hypothesis. Combining this with~\eqref{eq:distbnd3} gives the first inequality in~\eqref{I-ineq}. The second follows immediately by \cref{lem:ruzsa-multi}(ii).
\end{proof}
\begin{remark}%
  \label{rem:sum-fiber}
 The choices of the $X'_i, Y_i$ used in this argument are analogous to the choices we called ``sums'' and ``fibres'' in the informal discussion in~\cite[Section 3]{ggmt}.
\end{remark}

\subsection{The endgame}

We now define a tuple of independent random variables $(Y_{i,j})_{i,j \in \Z/m\Z}$ as follows:
by a slight abuse of notation, we identify $\Z/m\Z$ and $\{1,\dots,m\}$ in the obvious way, and let $Y_{i,j}$ be an independent copy of $X_i$.

We will be interested in the following random variables derived from $(Y_{i,j})_{i,j \in \Z/m\Z}$:
\[
  W \coloneqq \sum_{i,j \in \Z/m\Z} Y_{i,j}
\]
and
\[
  Z_1 \coloneqq \sum_{i,j \in \Z/m\Z} i Y_{i,j},\ \ \
  Z_2 \coloneqq \sum_{i,j \in \Z/m\Z} j Y_{i,j},\ \ \
  Z_3 \coloneqq \sum_{i,j \in \Z/m\Z} (-i-j) Y_{i,j}.
\]
The addition $(-i-j)$ takes place over $\Z/m\Z$.
Note that, because we are assuming $G$ is $m$-torsion, it is well-defined to multiply elements of $G$ by elements of $\Z/m\Z$.
Moreover, we note that the $Z_i$ are all supported on $m^3 S$.

Because they will arise frequently, we will also define for $i,j,r \in \Z/m\Z$ the variables
\begin{equation}\label{pqr-defs}
  P_i \coloneqq \sum_{j \in \Z/m\Z} Y_{i,j} , \quad
  Q_j \coloneqq \sum_{i \in \Z/m\Z} Y_{i,j} , \quad
  R_r \coloneqq \sum_{\substack{i,j \in \Z/m\Z \\ i+j=-r}} Y_{i,j} .\end{equation}
Note the identities
\begin{equation}
  \label{eq:is-fn}
  Z_1 = \sum_{i \in \Z/m\Z} i P_i,\qquad
  Z_2 = \sum_{j \in \Z/m\Z} j Q_j,\qquad
  Z_3 = \sum_{r \in \Z/m\Z} r R_r.
\end{equation}

There are several key facts to note concerning this situation.
One is the easily verified statement that
\begin{equation}%
  \label{eq:sum-zero}
  Z_1+Z_2+Z_3= 0
\end{equation}
holds identically.
Another is the following statement, which roughly says that $Z_1,Z_2,Z_3$ are ``almost'' pairwise independent conditional on $W$.
\begin{proposition}%
  \label{prop:52}
  Assuming still that~\eqref{k-def} and~\eqref{main-contra} hold, we have
  \[
    \mutual{Z_1 : Z_2\, |\, W},\
    \mutual{Z_2 : Z_3\, |\, W},\
    \mutual{Z_1 : Z_3\, |\, W} \leq t
  \]
  where
  \begin{equation}\label{t-def}
    t \coloneqq 2m (2m+1) \eta k.
  \end{equation}
\end{proposition}
\begin{proof}%
  We analyze these variables by applying \cref{key} in several different ways.
  In the first application, take $X_{i,j}=Y_{i,j}$.
  Note that each column $(X_{i,j})_{i=1}^m$ is indeed a permutation of $X_1,\dots,X_m$; in fact, the trivial permutation.
  Note also that for each $i \in \Z/m\Z$, the row sum is
  \[
    \sum_{j=1}^m X_{i,j} = \sum_{j \in \Z/m\Z} Y_{i,j} = P_i
  \]
  and for each $j \in \Z/m\Z$, the column sum is
  \[
    \sum_{i=1}^m X_{i,j} = \sum_{i \in \Z/m\Z} Y_{i,j} = Q_j.
  \]
  Finally note that $\sum_{i,j=1}^m X_{i,j} = W$.
  The conclusion of \cref{key} then states that
  \[
    \Bigmutual{ (P_i)_{i \in \Z/m\Z} : (Q_j)_{j \in \Z/m\Z } \,\big|\, W } \leq t,
  \]
  with $t$ as in~\eqref{t-def}.
  Since $Z_1$ is a function of $(P_i)_{i \in \Z/m\Z}$ by~\eqref{eq:is-fn}, and similarly $Z_2$ is a function of $(Q_j)_{j \in \Z/m\Z}$, it follows immediately from the data processing inequality, \cref{lem:data-processing}, that
  \[
    \mutual{ Z_1 : Z_2 \,|\, W } \leq t.
  \]

  In the second application of \cref{key}, we instead consider $X'_{i,j} = Y_{i-j,j}$.
  Again, for each fixed $j$, the tuple $(X'_{i,j})_{i=1}^m$ is a permutation of $X_1,\dots,X_m$.
  This time the row sums for $i \in \{1,\dots, m\}$ are
  \[
    \sum_{j=1}^m X'_{i,j} = \sum_{j \in \Z/m\Z} Y_{i-j,j} = R_{-i}.
  \]
 Similarly, the column sums for $j \in \{1,\dots, m\}$ are
  \[
    \sum_{i=1}^m X'_{i,j} = \sum_{i \in \Z/m\Z} Y_{i-j,j} = Q_j.
  \]
  As before, $\sum_{i,j=1}^m X'_{i,j} = W$.
  Hence, using~\eqref{eq:is-fn} and the data processing inequality again, the conclusion of \cref{key} tells us
  \[
    \mutual{ Z_3 :  Z_2 \,|\, W} \leq \Bigmutual{ (R_i)_{i \in \Z/m\Z} : (Q_j)_{j \in \Z/m\Z } \,\big|\, W } \leq t.
  \]

  In the third application\footnote{In fact, by permuting the variables $(Y_{i,j})_{i,j \in \Z/m\Z}$, one can see that the random variables $(W, Z_1, Z_2)$ and $(W, Z_1, Z_3)$ have the same distribution, so this is in some sense identical to -- and can be deduced from -- the first application.} of \cref{key}, take $X''_{i,j} = Y_{i,j-i}$.
  The column and row sums are respectively
  \[
    \sum_{j=1}^m X''_{i,j} = \sum_{j \in \Z/m\Z} Y_{i,j-i} = P_i \] and
\[     \sum_{i=1}^m X''_{i,j} = \sum_{i \in \Z/m\Z} Y_{i,j-i} = R_{-j}.
  \]
  Hence, \cref{key} and data processing gives
  \[
    \mutual{ Z_1 : Z_3 \,|\, W} \leq \Bigmutual{ (P_i)_{i \in \Z/m\Z} : (R_j)_{j \in \Z/m\Z } \,\big|\, W } \leq t,
  \]
  which completes the proof.
\end{proof}

At this point, we are in a very similar situation to~\cite[Section~7]{ggmt}: conditioning on a typical value $W=w$, the random variables $Z_1,Z_2,Z_3$ are ``almost pairwise independent'', in the sense that the mutual information of any two of them is small.
Using the entropic Balog--Szemer\'edi--Gowers lemma, we can find related variables with very small doubling, and then use these as candidates in~\eqref{main-contra-cond} to obtain a contradiction (unless $k=0$).

To put this into practice, we first collect some Ruzsa calculus type estimates about the variables $W$ and $Z_2$.
\begin{lemma}%
  \label{lem:ruzsa-w-z}
  For $W$ and $Z_2$ as above, the following hold:
  \begin{enumerate}[label=\textup{(\roman*)}]
    \item $\ent{W} \leq (2m-1)k + \frac1m \sum_{i=1}^m \ent{X_i}$;
    \item $\ent{Z_2} \leq \big( 28 (m-1) \log_2 m\big) k + \frac{1}{m} \sum_{i=1}^m \ent{X_i}$;
    \item $\mutual{W : Z_2} \leq 2 (m-1) k$;
    \item $\sum_{i=1}^m \dist{X_i}{Z_2|W} \leq 15 (m^2 \log_2 m) k$.
  \end{enumerate}
\end{lemma}
\begin{proof}%
  Without loss of generality, we may take $X_1,\dots,X_m$ to be independent. Write $S = \sum_{i=1}^m X_i$.
  Note that for each $j \in \Z/m\Z$, the sum $Q_j$ from~\eqref{pqr-defs} above has the same distribution as $S$.
  By \cref{ruzsa-calculus}\ref{rc-x} we have
  \begin{align*}
    \ent{W} = \Bigent{\tsum_{j \in \Z/m\Z} Q_j}  & \leq \ent{S} + \sum_{j=2}^m (\ent{Q_1+Q_j} - \ent{S}) \\ & = \ent{S} + (m-1) \dist{S}{-S}.
  \end{align*}
  By \cref{lem:cool-dist-fact}(iii), we have
  \begin{equation}
    \label{eq:s-bound}
    \dist{S}{-S} \leq 2 k
  \end{equation}
  and hence
  \[
    \ent{W} \leq 2 k (m-1) + \ent{S}.
  \]
  Since, by the definition of multidistance,
  \begin{equation}
    \label{eq:ent-s}
    \ent{S} = k + \frac1m \sum_{i=1}^m \ent{X_i},
  \end{equation}
  this implies (i).

  Turning to (ii), we observe
  \[
    \ent{Z_2} = \Bigent{\tsum_{j \in \Z/m\Z} j Q_j}.
  \]
  Applying \cref{ruzsa-calculus}\ref{rc-x} again gives
  \begin{align*}
    \ent{Z_2} &\leq \sum_{i=2}^{m-1} \ent{Q_1 + i Q_i}  - (m-2) \ent{S}.
  \end{align*}
  Using \cref{lem-lin}(ii) and~\eqref{eq:s-bound} we get
  \begin{align*}
    \ent{Z_2}
              &\leq \ent{S} + (10 \lfloor \log_2 m \rfloor + 4)(m-2) \dist{S}{-S} \\
              &\leq \ent{S} + (20 \lfloor \log_2 m \rfloor + 8) (m-2) k.
  \end{align*}
  Applying~\eqref{eq:ent-s} (and crude estimates to tidy the terms involving $m$) proves (ii). We remark that using the weaker but easier \cref{lem-lin}(i) in place of \cref{lem-lin}(ii) here would lead to similar bounds in our main results.

  Turning to (iii), we of course have  $\mutual{W : Z_2} = \ent{W} - \ent{W | Z_2}$, and since $Z_2 = \sum_{j=1}^{m-1} j Q_j$ and $W = \sum_{j=1}^m Q_j$,
  \[
    \ent{W | Z_2} \geq \ent{W \,|\, Q_1,\dots,Q_{m-1}} = \ent{Q_m} = \ent{S}.
  \]
  Hence, by (i) and~\eqref{eq:ent-s},
  \[
    \mutual{W : Z_2} \leq \ent{W} - \ent{S} \leq 2 (m-1) k,
  \]
 which is (iii).

 Finally we turn to (iv). For each $i \in \{1,\dots, m\}$, using \cref{lem:cool-dist-fact}(i) (noting the sum $Z_2$ contains $X_i$ as a summand) we have
  \begin{equation}\label{in-a-bit-6}
    \dist{X_i}{Z_2} \leq \dist{X_i}{X_i} + \tfrac12 (\ent{Z_2} - \ent{X_i})
  \end{equation}
  and using \cref{ruzsa-calculus}\ref{rc-viii} we have
  \[
    \dist{X_i}{Z_2 | W} \leq \dist{X_i}{Z_2} + \tfrac12 \mutual{W : Z_2}.
  \]
 Combining with~\eqref{in-a-bit-6} and (iii) gives
 \[ \dist{X_i}{Z_2 | W} \leq \dist{X_i}{X_i} + \tfrac12 (\ent{Z_2} - \ent{X_i}) + (m-1)k.\]
 Summing over $i$ and applying (ii) gives
 \[ \sum_{i = 1}^m \dist{X_i}{Z_2 | W} \leq \sum_{i = 1}^m \dist{X_i}{X_i} + \frac{m}{2} (28 (m-1) \log_2 m) k + m(m-1) k.\]
Finally, applying \cref{lem:ruzsa-multi}(ii) (and crude bounds for the terms involving $m$) gives (iv).
\end{proof}

We next prove the following analogue of~\cite[Lemma~7.2]{ggmt}.
\begin{lemma}%
  \label{lem:get-better}
  Let $G$ be an abelian group, let $(T_1,T_2,T_3)$ be a $G^3$-valued random variable such that $T_1+T_2+T_3=0$ holds identically, and write
  \[
    \delta \coloneqq \mutual{T_1 : T_2} + \mutual{T_1 : T_3} + \mutual{T_2 : T_3}.
  \]
  Let $Y_1,\dots,Y_n$ be some further $G$-valued random variables and let $\alpha>0$ be a constant.
  Then there exists a random variable $U$, with the support of $U$ contained in that of $T_2$, such that
  \begin{equation}
    \label{eq:get-better}
    \dist{U}{U} + \alpha \sum_{i=1}^n \dist{Y_i}{U} \leq \Bigl(2 + \frac{\alpha n}{2} \Bigr) \delta + \alpha \sum_{i=1}^n \dist{Y_i}{T_2}.
  \end{equation}
\end{lemma}
\begin{proof}%
  Apply entropic Balog--Szemer\'edi--Gowers (\cref{ruzsa-calculus}\ref{rc-iv}) with $X=T_1$ and $Y=T_2$.
  Since $T_1+T_2=-T_3$, we find that
  \begin{align}\nonumber
    \sum_{z} p_{T_3}(z) & \dist{T_1 \,|\, T_3 \mathop{=} z}{T_2 \,|\, T_3 \mathop{=} z} \\ \nonumber
    &\leq  3 \mutual{T_1 : T_2} + 2 \ent{T_3} - \ent{T_1} - \ent{T_2} \\
    &\ = \mutual{T_1 : T_2} + \mutual{T_1 : T_3} + \mutual{T_2 : T_3}
    = \delta,\label{514a}
  \end{align}
  where the last line follows by observing
  \[
    \ent{T_1,T_2} = \ent{T_1,T_3} = \ent{T_2,T_3} = \ent{T_1,T_2,T_3}
  \]
  since any two of $T_1,T_2,T_3$ determine the third, and by unpacking the definition of mutual information.

  By~\eqref{514a} and the triangle inequality,
  \[
    \sum_z p_{T_3}(z) \dist{T_2 \,|\, T_3 \mathop{=} z}{T_2 \,|\, T_3\mathop{=}z} \leq 2 \delta
  \]
  and by \cref{ruzsa-calculus}\ref{rc-viii}, for each $Y_i$,
  \begin{align*}
    &\sum_z p_{T_3}(z) \dist{Y_i}{T_2 \,|\, T_3 \mathop{=} z} \\
    &\qquad= \dist{Y_i}{T_2 \,|\, T_3}
                                    \leq \dist{Y_i}{T_2} + \frac12 \mutual{T_2 : T_3}
                                    \leq \dist{Y_i}{T_2} + \frac{\delta}{2}.
  \end{align*}
  Hence,
  \begin{align*}
    &\sum_z p_{T_3}(z) \bigg( \dist{T_2 \,|\, T_3 \mathop{=} z}{T_2 \,|\, T_3 \mathop{=} z} + \alpha \sum_{i=1}^n \dist{Y_i}{T_2 \,|\, T_3 \mathop{=} z} \bigg) \\
    &\qquad \leq
    \Bigl(2 + \frac{\alpha n}{2} \Bigr) \delta + \alpha \sum_{i=1}^n \dist{Y_i}{T_2},
  \end{align*}
  and the result follows by setting $U=(T_2 \,|\, T_3 \mathop{=} z)$ for some $z$ such that the quantity in parentheses on the left-hand side is at most the weighted average value.
\end{proof}

Finally, we can put this all together.
For each value $W=w$, apply \cref{lem:get-better} to
\[
  T_1 = (Z_1 \,|\, W \mathop{=} w),\
  T_2 = (Z_2 \,|\, W \mathop{=} w),\
  T_3 = (Z_3 \,|\, W \mathop{=} w)
\]
with $Y_i=X_i$ and $\alpha=\eta/m$, with $\eta$ the constant in the statement of \cref{thm-main-multi-dec}.
Write
\[
  \delta_w \coloneqq \mutual{T_1 : T_2} + \mutual{T_1 : T_3} + \mutual{T_2 : T_3}
\]
for this choice, and note that
\begin{align}
  \label{eq:delta-w}
  \nonumber
  \delta_\ast \coloneqq \sum_w p_W(w) \delta_w &= \mutual{Z_1 : Z_2 \,|\, W} + \mutual{Z_1 : Z_3 \,|\, W} + \mutual{Z_2 : Z_3 \,|\, W} \\
                         & \leq 6 m(2m+1) \eta k \ll \eta m^2 k
\end{align}
by \cref{prop:52}.
Write $U_w$ for the random variable guaranteed to exist by \cref{lem:get-better},
so that~\eqref{eq:get-better} gives \begin{equation}
  \label{eq:uw1}
  \dist{U_w}{U_w} \leq \Bigl( 2 + \frac{\alpha m}{2} \Bigr) \delta_w + \alpha \sum_{i=1}^m \bigl( \dist{X_i}{T_2} - \dist{X_i}{U_w} \bigr).
\end{equation}
Note that (by \cref{lem:get-better}) the support of $U_w$ is contained in that of $Z_2$. Recalling that $Z_2 = \sum_{i,j \in \Z/m\Z} j Y_{i,j}$, with the $Y_{i,j}$ all being copies of $X_i$ and hence supported on $S$, we see that the support of $U_w$ is contained in $m^3 S$.

Let $(U_w)_I$ denote the tuple consisting of the same variable $U_w$ repeated $m$ times.
By \cref{lem:ruzsa-multi}(iii),
\begin{equation}
  \label{eq:uw2}
  \multidist{(U_w)_I} \leq m \dist{U_w}{U_w}.
\end{equation}
On the other hand, applying our hypothesis~\eqref{main-contra} gives
\begin{equation}
  \label{eq:uw3}
  \multidist{(U_w)_I} \geq (1-\eta) k - \eta \sum_{i=1}^m \dist{X_i}{U_w}.
\end{equation}
Combining~\eqref{eq:uw1},~\eqref{eq:uw2} and~\eqref{eq:uw3} and averaging over $w$ (with weight $p_W(w)$), and recalling the value $\alpha=\eta/m$, gives
\[
   m \Bigl( 2 + \frac{\eta}{2} \Bigr) \delta_\ast + \eta \sum_{i=1}^m \dist{X_i}{Z_2 | W}
  \geq (1 - \eta) k
\]
since the terms $\dist{X_i}{U_w}$ cancel by our choice of $\alpha$.
Substituting in \cref{lem:ruzsa-w-z}(iv) and~\eqref{eq:delta-w}, and using the fact that $2 + \frac{\eta}{2} < 3$, we have
\[
  m^3 \eta k + \eta (m^2 \log_2 m ) k \gg k.
\]
Recall that, in the statement of \cref{thm-main-multi-dec}, $\eta$ was taken to be $c/m^3$. If the constant $c$ is sufficiently small, this gives a contradiction unless $k=0$.
This is what we needed to prove, and the proof of \cref{thm-main-multi-dec} is complete.

\appendix

\section{Entropic Ruzsa calculus}\label{entropy-app}

In this appendix we collect a number of useful inequalities involving entropic Ruzsa distance and related quantities.

\begin{proposition}\label{ruzsa-calculus}
  Let $G$ be an abelian group, and let $X,Y,Z$ be $G$-valued random variables. Then we have the following statements.
  \begin{enumerate}[label=\textup{(\roman*)}]
    \item\label{rc-i}%
      $\dist{X}{Y} = \dist{Y}{X} \geq 0$ and $\dist{X}{Z} \leq \dist{X}{Y} + \dist{Y}{Z}$.
    \item\label{rc-ii}%
      $\max( \ent{X}, \ent{Y} ) - \mutual{X:Y} \leq \ent{X \pm Y}$ for either choice of sign $\pm$.  In particular, if $X,Y$ are independent, then $\ent{X}$, $\ent{Y}$ are both at most $\ent{X \pm Y}$.
    \item\label{rc-iv}%
      \uppar{Entropic Balog--Szemer\'edi--Gowers}
      \begin{align*} \sum_z p_{X+Y}&(z) \dist{X|X+Y=z}{Y|X+Y=z}\\ &  \leq 3 \mutual{X:Y} + 2\ent{X+Y} - \ent{X} - \ent{Y}.\end{align*}
    \item\label{rc-v}%
      If $(X,Z)$ are independent of $Y$ \uppar{but $X$ and $Z$ are not necessarily independent of each other}, then
      \[  \ent{X-Z} \leq \ent{X-Y} + \ent{Y-Z} - \ent{Y}. \]
    \item\label{rc-vii}%
      $\dist{X}{-Y} \leq 3 \dist{X}{Y}$.
    \item\label{rc-viii}%
      $\dist{X}{Y|Z} \leq \dist{X}{Y} + \tfrac{1}{2} \mutual{Y:Z}$.
    \item\label{rc-ix}%
      If $Y,Z$ are independent, then
      \[ \dist{X}{Y|Y+Z} \leq \dist{X}{Y} + \tfrac{1}{2} (\ent{Y+Z} - \ent{Z}).\]
    \item\label{rc-x}%
      If $X,Y_1,\dots,Y_n$ are jointly independent random variables, then
      \begin{equation}\label{kv-1}
        \Bigent{ X + \tsum_{i=1}^n Y_i } - \ent{X}  \leq \sum_{i=1}^n \bigl( \ent{X+Y_i}- \ent{X} \bigr)
      \end{equation} and
      \begin{equation}\label{ruzsa-2}
        \Bigdist{X}{\tsum_{i = 1}^n Y_i} \leq 2 \sum_{i=1}^n \dist{X}{Y_i}.
      \end{equation}
  \end{enumerate}
\end{proposition}

\begin{proof}  For~\ref{rc-i}--\ref{rc-iv}, see~\cite[Appendix A]{ggmt}.  For~\ref{rc-v}, see~\cite[Lemma 1.1 (i)]{gmt} which, after unpacking the definitions there, is the same statement.
  For~\ref{rc-vii}, see~\cite[Theorem 1.10]{tao-entropy}.
  For~\ref{rc-viii}, see~\cite[Lemma 5.1]{ggmt}.
  To prove~\ref{rc-ix} (which was also established in~\cite[Lemma 5.2]{ggmt}), we apply part~\ref{rc-viii} and calculate
\begin{align*}
  \mutual{Y:Y+Z}
  &= \ent{Y
  } + \ent{Y+Z} - \ent{Y,Y+Z} \\
  &= \ent{Y} + \ent{Y+Z} - \ent{Y,Z}\\
  &= \ent{Y+Z} - \ent{Z}.
\end{align*}

Now we establish~\ref{rc-x}, which is essentially due to Madiman~\cite{madiman}, and related to an earlier inequality of Kaimanovich and Vershik~\cite{kaimanovich-vershik}.
The inequality~\eqref{kv-1} follows easily by induction from the $n = 2$ case, which is~\cite[Lemma A.1]{ggmt}.
We turn now to~\eqref{ruzsa-2}.
From the definition of distance, and replacing all $Y_i$ by $-Y_i$, we may rewrite~\eqref{kv-1} as
  \begin{align*}
  \bigdist{X}{\tsum_{i=1}^n Y_i} + \tfrac{1}{2}  \Bigl(& \Bigent{\tsum_{i = 1}^n Y_i}-\ent{X}\Bigr) \\ & \leq \sum_{i=1}^n \Bigl(\dist{X}{Y_i} + \tfrac{1}{2} \bigl(\ent{X}- \ent{Y_i}\bigr)\Bigr).\end{align*}
  However, by the nonnegativity of multidistance we have $\bigent{\sum_{i=1}^n Y_i} \geq \tfrac{1}{n} \sum_{i = 1}^n \ent{Y_i}$, so after rearranging we obtain
    \[
    \bigdist{X}{\tsum_{i=1}^n Y_i} \leq \sum_{i=1}^n \dist{X}{Y_i} + \frac{n-1}{2n} \sum_{i=1}^n\bigl(\ent{Y_i}- \ent{X}\bigr).
  \]
  The claimed bound~\eqref{ruzsa-2} now follows using the inequality $\ent{Y_i} - \ent{X} \leq 2 \dist{X}{Y_i}$ (see~\cite[(A.12)]{ggmt}); in fact we can replace the constant $2$ by the slightly sharper constant $(2n-1)/n$.
\end{proof}

In a similar spirit, we give some further inequalities relating to multidistance.
\begin{lemma}%
  \label{lem:cool-dist-fact}
  Let $G$ be an abelian group,
  suppose $I$ is a finite index set, $|I| \geq 2$, and that $(X_i)_{i \in I}$ are jointly independent $G$-valued random variables.
  \begin{enumerate}[label=\textup{(\roman*)}]
    \item For any further random variable $Y$, and any $i_0 \in I$, we have
      \[
        \Bigdist{Y}{\tsum_{i \in I} X_i} \leq \dist{Y}{X_{i_0}} + \frac12 \Bigl( \Bigent{\tsum_{i \in I} X_i} - \ent{X_{i_0}} \Bigr).
      \]
    \item For any finite index set $J$, any jointly independent random variables $(Y_j)_{j \in J}$ independent of the $(X_i)_{i \in I}$, and any function $f \colon J \to I$, we have
      \[
        \Bigent{\tsum_{j \in J} Y_j} \leq \
        \Bigent{\tsum_{i \in I} X_i} + \sum_{j \in J} \bigl( \ent{Y_j - X_{f(j)}} - \ent{X_{f(j)}} \bigr).
      \]
    \item If we write $W = \sum_{i \in I} X_i$ then $\dist{W}{-W} \leq 2 \multidist{X_I}$.
  \end{enumerate}
  \end{lemma}
\begin{proof}%
  We may assume without loss of generality that $Y$, $(Y_j)_{j \in J}$ and $X_I$ are all independent of each other.
  Applying~\eqref{kv-1} with $n=2$, $X=X_{i_0}$, $Y_1=-Y$ and $Y_2=\sum_{i \ne i_0} X_i$ gives
  \begin{equation}
    \label{eq:cool-dist-fact}
    \Bigent{-Y + \tsum_{i \in I} X_i} \leq \ent{X_{i_0}-Y} + \Bigent{\tsum_{i \in I} X_i} - \ent{X_{i_0}}.
  \end{equation}
  Using the definition of Ruzsa distance, this gives (i).

  For (ii), write $W \coloneqq \sum_{i \in I} X_i$.
  Then
  \begin{align*}
    \Bigent{\tsum_{j \in J} Y_j} &\leq \Bigent{-W + \tsum_{j \in J} Y_j} \\
    &\leq \ent{W} + \sum_{j \in J} \bigl(\ent{Y_j-W} - \ent{W}\bigr) \\
    &\leq \ent{W} + \sum_{j \in J} \bigl( \ent{Y_j - X_{f(j)}} - \ent{X_{f(j)}} \bigr)
  \end{align*}
  where in the penultimate step we used~\eqref{kv-1} and in the last step we used~\eqref{eq:cool-dist-fact}.

  For (iii), take $(X'_i)_{i \in I}$ to be further independent copies of $(X_i)_{i \in I}$ and write $W' = \sum_{i \in I} X'_i$.
  Fix any distinct $a,b \in I$.

  Applying~\eqref{kv-1} with $n = 2$, $X = X_a$, $Y_1 = \sum_{i \neq a} X_i$ and $Y_2 = W'$, we obtain
  \begin{equation}\label{7922}
    \ent{W + W'} \leq
    \ent{W} + \ent{X_{a} + W'} - \ent{X_{a}}.
   \end{equation}
   Applying~\eqref{kv-1} with $n = 2$, $X = X'_b$, $Y_1 = X_a$ and $Y_2 = \sum_{i \neq b} X'_i$ gives
   \[ \ent{X_a + W'} \leq \ent{X_a + X_b} + \ent{W'} - \ent{X'_b}.\]
   Combining this with~\eqref{7922} and then applying \cref{ruzsa-calculus}\ref{rc-ii} gives
   \begin{align*}  \ent{W + W'}  & \leq    2\ent{W} + \ent{X_a + X_b}  - \ent{X_a} - \ent{X_b} \\ & \leq
    3 \ent{W} - \ent{X_a} - \ent{X_b}.
  \end{align*}
  Averaging this over all choices of $(a,b)$ gives $\ent{W} + 2 \multidist{X_I}$, and rearranging gives (iii).
\end{proof}

Next, we obtain an
entropic analogue of some estimates of Bukh~\cite{bukh} on sums of dilates.
\begin{lemma}\label{lem-lin}  Let $G$ be an abelian group, let $X,Y$ be independent $G$-valued random variables, and let $a \in \Z$. Then we have the following two inequalities.
\begin{enumerate}[label=\textup{(\roman*)}]
  \item $\ent{X - aY} - \ent{X} \leq 4 |a| \dist{X}{Y}$.
  \item $\ent{X - aY} - \ent{X} \leq (4 + 10 \lfloor \log_2 |a|\rfloor) \dist{X}{Y}$.
\end{enumerate}
\end{lemma}
\begin{proof}%
  Let $a$ be any integer, and let $X'$ be another independent copy of $X$.
  We will show two key inequalities:
  \begin{equation}
    \label{eq:ineq1}
    \ent{X - (a \pm 1) Y} \leq \ent{X - a Y} + \bigl(\ent{X-Y-X'} - \ent{X}\bigr)
  \end{equation}
  and
  \begin{equation}
    \label{eq:ineq2}
    \ent{X - 2 a Y} \leq \ent{X - a Y} +  \bigl(\ent{X - 2X'} - \ent{X}\bigr).
  \end{equation}
  We first show why these suffice.
  By~\eqref{kv-1} with $n = 2$ we have
  \[ \ent{Y - X + X'} \leq \ent{Y - X} + \ent{Y + X'} - \ent{Y},\] which rearranges to give
  \[ \ent{X-Y-X'} - \ent{X} \leq \dist{X}{Y}+\dist{X}{-Y}.
  \] By \cref{ruzsa-calculus}\ref{rc-vii}, this is at most $4 \dist{X}{Y}$, and so~\eqref{eq:ineq1} gives
    \begin{equation}\label{f-recur-1}
    \ent{X - (a \pm 1) Y} \leq \ent{X - a Y} + 4 \dist{X}{Y}.
  \end{equation}
  By a straightforward induction on $a$, this proves (i).

  Now let $X''$ be a further independent copy of $X$.  By the case $a = 1$ of~\eqref{eq:ineq1} applied to $X, X'$ we obtain
  \[
    \ent{X-2X'} \leq \ent{X-X'} + \ent{X-X'-X''} - \ent{X},
  \]
  and another application of~\eqref{kv-1} with $n = 2$ gives
  \[
    \ent{X-X'-X''} \leq \ent{X-X'} + \ent{X-X''} - \ent{X} = \ent{X} + 2 \dist{X}{X}.
  \]
  Hence,
  \[
    \ent{X-2X'} \leq \ent{X} + 3 \dist{X}{X} \leq \ent{X} + 6 \dist{X}{Y}
  \]
  where in the second bound we use the triangle inequality. Hence,~\eqref{eq:ineq2} applied to $X$ and $Y$ implies
  \begin{equation}\label{f-recur-2}
    \ent{X - 2 a Y} \leq \ent{X - a Y} + 6 \dist{X}{Y}.
  \end{equation}
 Inequalities~\eqref{f-recur-1} and~\eqref{f-recur-2} imply that we may recursively bound
  \[
    \ent{X - a Y} - \ent{X} \leq f(a) \dist{X}{Y},
  \]
  where $f \colon \Z \to \Z_{\ge 0}$ is the smallest function obeying $f(0) = 0$, $f(a\pm 1) \leq f(a)+4$ and $f(2a) \leq f(a) + 6$.
  Certainly, we may insert a new binary digit at the start of $a$ at the cost of increasing $f(a)$ by at most $10$, which gives the bound
  $f(a) \leq 4 + 10 \lfloor \log_2 |a| \rfloor$,  as required for (ii).

  It remains to establish~\eqref{eq:ineq1} and~\eqref{eq:ineq2}.
  For the former, we write $X-(a+1)Y$ as $(X-Y) - aY$ and apply \cref{ruzsa-calculus}\ref{rc-v} with the variables $(X,Y,Z)$ there replaced by $(X-Y, X', aY)$ to obtain
  \[
    \ent{X-Y-aY} \leq \ent{X-Y-X'} + \ent{X'-aY} - \ent{X'}.
  \]
  On relabeling this is exactly~\eqref{eq:ineq1} in the case $a+1$.
  For $a-1$, we simply replace all terms $X-Y$ with $X+Y$, noting that $\ent{X-X'+Y}=\ent{X-X'-Y}$.

  For~\eqref{eq:ineq2} we apply the triangle inequality \cref{ruzsa-calculus}\ref{rc-i} (or if you prefer, \cref{ruzsa-calculus}\ref{rc-v} again) to get
  \[
    \ent{X-2aY} \leq \ent{X-2X'} + \ent{2X'-2aY} - \ent{2X'}.
  \]
 It thus suffices to show that
  \begin{equation}\label{need-a6}
    \ent{2X'-2aY} - \ent{2X'} \leq \ent{X'-aY} - \ent{X'}.
  \end{equation}
  To see this, observe that by submodularity we have
  \begin{align*}
   \ent{X'-aY \,|\, 2X'-2aY} & \geq \ent{X'-aY \,|\, 2X'-2aY, aY} \\ & = \ent{X' \,|\, 2X'-2aY, aY} \\ & = \ent{X' \,|\, 2X', aY} =  \ent{X' \,|\, 2 X'},\end{align*}
  and this rearranges to give the claimed inequality~\eqref{need-a6}.
  This completes the proof.
\end{proof}

Finally, for reference we recall the \emph{data processing inequality}, a standard result on mutual information.
\begin{lemma}%
  \label{lem:data-processing}
  Let $X,Y,Z$ be random variables.  For any functions $f,g$ on the ranges of $X,Y$ respectively, we have $\mutual{ f(X):g(Y) | Z} \leq \mutual{ X:Y|Z }$.
\end{lemma}
\begin{proof}  It suffices to prove the unconditional version $\mutual{f(X):g(Y)} \leq \mutual{X:Y}$ of this inequality, as the conditional version then follows by conditioning to the events $Z=z$, multiplying by $p_Z(z)$, and summing over $z$.  By symmetry and iteration it will suffice to prove the one-sided data processing inequality $\mutual{f(X):Y} \leq \mutual{X:Y}$, or equivalently that $\ent{Y|X} \leq \ent{Y|f(X)}$.  But as $X$ determines $f(X)$, we have $\ent{Y|X} = \ent{Y|f(X),X}$, and the claim follows from submodularity~\eqref{nonneg-cond}.
\end{proof}

\section{From entropic PFR to combinatorial PFR}\label{comb-app}

In this appendix we repeat the arguments from~\cite{gmt} (which were specialized to the $\F_2$ case) to derive \cref{mainthm} from \cref{main-entropy}.

Let $m,A,K$ be as in \cref{mainthm}.  By translation we may assume without loss of generality that $A$ contains $0$.
Let $U_A$ be the uniform distribution on $A$.  The doubling condition $|A+A| \leq K|A|$ and Jensen's inequality give
\[  \dist{U_A}{-U_A} \leq \log K. \]
By \cref{main-entropy} (with $S$ equal to $A \cup -A$, which is clearly symmetric and contains $0$), we may thus find a subspace $H$ of $G$, with $H \subseteq \ell S \subseteq \ell A - \ell A$ for some $\ell \ll (2 + m\log K)^{O(m^3 \log m)}$, such that
\[  \dist{U_A}{U_H} \ll  m^3 \log K. \]
Now, since $\tfrac{1}{2} |\ent{X} - \ent{Y}| \leq \dist{X}{Y}$ we conclude that
\[  \log |H| = \log |A| + O( m^3 \log K ) \]
and also
\[  \ent{U_A - U_H} = \log |H| + O( m^3 \log K ). \]
Applying~\cite[(A.2)]{gmt}, we conclude the existence of a point $x_0 \in \F_p^n$ such that
\[  p_{U_A-U_H}(x_0) \geq e^{-\ent{U_A - U_H}} \geq K^{-O(m^3)}/ |H|, \]
or in other words
\[  |A \cap (H + x_0)| \geq K^{-O(m^3)} |H|. \]
Applying the Ruzsa covering lemma~\cite[Lemma 2.14]{tao-vu}, we may thus cover $A$ by at most $K^{O(m^3)}$ translates of
\[  (A \cap (H + x_0) - A \cap (H + x_0)) \subseteq H. \]
By subdividing $H$ into cosets $H'$ of a subgroup of cardinality between $|A|/m$ and $|A|$ if necessary, the claim then follows.

\section{Inverse theorem for \texorpdfstring{$U^3$}{U3}}\label{u3-app}

In this appendix we sketch a derivation of \cref{inverse-cor} from \cref{mainthm}.  In the even characteristic case $p=2$ this implication was worked out in~\cite{gt-equiv,lovett,samorodnitsky}.  In odd characteristic, the implication is almost worked out in~\cite{gt-inverseu3}, except that that argument contained a somewhat gratuitous invocation of the Bogolyubov theorem, which currently does not have polynomial bounds and so does not recover the full strength \cref{inverse-cor}.  However, it is not difficult to modify the argument in~\cite{gt-inverseu3} to avoid the use of Bogolyubov's theorem, and we do so here.  We will assume familiarity with the notation in that paper.

Let the notation be as in \cref{inverse-cor}.  Applying~\cite[Proposition 5.4]{gt-inverseu3} (and identifying $\F_p^n$ with its Pontryagin dual $\widehat{\F_p^n}$ in the standard fashion), one can find a subset $H'$ of $\F_p^n$, a function $\xi : H' \to \F_p^n$ whose graph $\Gamma' \coloneqq \{ (h,\xi_h): h \in H' \} \subseteq \F_p^n \times \F_p^n$ obeys the estimates
\[  |\Gamma'| \gg \eta^{O(1)} p^n \]
and
\[  |2\Gamma'| \ll \eta^{-O(1)} p^n, \]
and such that
\[  |\E_{x \in \F_p^n} f(x+h) \overline{f(x)} e_p(-\xi_h \cdot x)| \gg \eta^{O(1)} \]
for all $(h,\xi_h) \in \Gamma'$, where $\cdot \colon \F_p^n \times \F_p^n \to \F_p$ is the usual inner product.

Applying \cref{mainthm}, we may cover $\Gamma'$ by $O(\eta^{-O(p^3)})$ translates of a subgroup $H$ of $\F_p^n \times \F_p^n$ of cardinality at most $|H'| \leq p^n$.  If we introduce the groups
\[  H_1 \coloneqq \{ y \in \F_p^n: (0,y) \in H \} \]
and
\[  H_0 \coloneqq \{ x \in \F_p^n: (x,y) \in H \hbox{ for some } y \in \F_p^n \} \]
then $|H| = |H_0| |H_1|$, and $H'$ can be covered by $O(\eta^{-O(p^3)})$ translates of $H_0$, hence
\[  |H_0| \gg \eta^{O(p^3)} p^n \]
and hence
\[  |H_1| = |H|/|H_0| \ll \eta^{-O(p^3)}. \]
By considering a complementing subspace of $0 \times H_1$ in $H$, we may write
\[  H = \{ (x,y): x \in H_0; y - M_0x \in H_1 \} \]
for some linear transformation $M_0 \colon H_0 \to \F_p^n$, which we may then extend (somewhat arbitrarily) to a homomorphism from $\F_p^n$ to $\F_p^n$.  As we are in odd characteristic, we can write $M_0=2M$ for some other linear map $M \colon \F_p^n \to \F_p^n$. One can then cover $H$ by $O(|H_1|) = O(\eta^{-O(p^3)})$ translates of the graph $\{ (x, 2Mx): x \in \F_p^n\}$, and hence $\Gamma'$ can also be covered by
$O(\eta^{-O(p^3 )})$ translates of this graph.  We conclude that
\[  \E_{h \in \F_p^n} 1_H(h) 1_{\xi_h = 2Mh + \xi_0} \gg \eta^{O(p^3 )} \]
for some $\xi_0 \in \F_p^n$.  This is a variant of the conclusion of~\cite[Proposition 6.1]{gt-inverseu3}, but without the restriction to the subspace $V$ (and with slightly worse exponents).  Repeating the ``symmetry argument'' in~\cite[\S 6, Step 2]{gt-inverseu3}, one can then find a subspace $W$ of $\F_p^n$ with
\begin{equation}\label{wpn}
  |W|/p^n \gg \eta^{O(p^3)}
\end{equation}
such that $M w \cdot w' = M w' \cdot w$ for all $w,w' \in W$.  Repeating the arguments in~\cite[\S 6, Step 3]{gt-inverseu3}, one then concludes that
\[  \E_{y \in \F_p^n} \|f\|_{u^3(y+W)} \gg \eta^{O(p^3)}. \]
By\footnote{This result relies on~\cite[Corollary 3.2]{gt-inverseu3}.  As pointed out very recently in~\cite[Remark 4.3]{cgss}, this corollary is false in the generality stated, but holds when $W$ is a complemented subgroup.  Fortunately, in $\F_p^n$, all subgroups are complemented, so this is not an issue for the current argument.}~\cite[Theorem 2.3(ii)]{gt-inverseu3} we have
\[  \|f\|_{u^3(\F_p^n)} \geq p^{-n} |W| \| f\|_{u^3(y+W)} \]
for all $y \in \F_p^n$, hence by~\eqref{wpn}
\[  \|f\|_{u^3(\F_p^n)} \gg \eta^{O(p^3)}. \]
By the definition of the $u^3$ norm, this gives the claim.


\begin{thebibliography}{99}

\bibitem{bukh}
B. Bukh, \emph{Sums of dilates}, Comb. Probab. Comput. \textbf{17} (2008), No. 5, 627--639.

\bibitem{cgss}
P. Candela, D. Gonz\'alez-S\'anchez, B. Szegedy, \emph{On the inverse theorem for Gowers norms in abelian groups of bounded torsion}, preprint, arXiv:2311.13899.

\bibitem{ggmt}
W.~T.~Gowers, B.~J.~Green, F.~Manners and T.~C.~Tao, \emph{On a conjecture of Marton}, preprint, arXiv:2311.05762.


\bibitem{gmt}
B.~J.~Green, F.~Manners and T.~C.~Tao, \emph{Sumsets and entropy revisited}, preprint, arXiv:2306.13403.



\bibitem{gt-inverseu3}
B.~J.~Green and T.~C.~Tao,  \emph{An inverse theorem for the Gowers $U^3(G)$-norm},
Proc. Edinb. Math. Soc. \textbf{51} (2008), no. 1, 73--153.

\bibitem{gt-equiv}
B.~J.~Green and T.~C.~Tao,  \emph{An equivalence between inverse sumset theorems and inverse conjectures for the $U^3$ norm}, Math. Proc. Cambridge Philos. Soc. \textbf{149} (2010), no.1, 1--19.



\bibitem{kaimanovich-vershik}
V.~A.~Kaimanovich and A.~M.~Vershik, \emph{Random walks on discrete groups: boundary and entropy}, Ann. Probab. \textbf{11}(1983), no.3, 457--490.

\bibitem{jjl}  
  J.-J.~Liao, \emph{Improved Exponent for Marton's Conjecture in $\F_2^n$}, preprint available at \url{https://arxiv.org/abs/2404.09639}.



\bibitem{lovett}
S.~Lovett, \emph{Equivalence of polynomial conjectures in additive combinatorics},
Combinatorica \textbf{32} (2012), no. 5, 607--618.


\bibitem{madiman}
M. Madiman, \emph{On the entropy of sums}, in Information Theory Workshop 2008. ITW '08.  IEEE: 303--307, 2008. 

\bibitem{hanson-petridis}
B.~Hanson, G.~Petridis, \emph{A question of Bukh on sums of dilates}, Discrete Anal. \textbf{2021} (2021), Paper No. 13, 21 p..

\bibitem{petridis}
G.~Petridis, \emph{New proofs of Pl\"unnecke-type estimates for product sets in groups},
Combinatorica \textbf{32} (2012), No. 6, 721--733.

\bibitem{plunnecke}
H.~Pl\"unnecke, \emph{Eine zahlentheoretische anwendung der graphtheorie}, J. Reine Angew. Math. \textbf{243} (1970), 171--183.

\bibitem{ruzsa}
I.~Z.~Ruzsa, \emph{An application of graph theory to additive number theory}, Sci. Ser. A Math. Sci. (N.S.) \textbf{3} (1989), 97--109.

\bibitem{ruzsa-entropy} I.~Z.~Ruzsa, \emph{Sumsets and entropy,} Random Struct. Alg., \textbf{34} (2009), 1--10.

\bibitem{samorodnitsky} A.~Samorodnitsky, \emph{Low-degree tests at large distances,} In Proceedings of the thirty-ninth annual ACM symposium on Theory of computing, 2007, 506--515.

\bibitem{sanders}
T.~Sanders, \emph{On the Bogolyubov--Ruzsa lemma}, Anal. PDE \textbf{5} (2012), 627 -- 655.

\bibitem{sanders2} T.~Sanders, \emph{The structure theory of set addition revisited,} Bull. Amer. Math. Soc. \textbf{50} (2013), 93--127.

\bibitem{tao-entropy}
T.~C.~Tao, \emph{Sumset and inverse sumset theory for Shannon entropy}, Combin. Probab. Comput. \textbf{19} (2010), no. 4, 603--639.

\bibitem{tao-vu}
T.~C. Tao and V. Vu, Additive combinatorics, Cambridge Stud. Adv. Math., 105
Cambridge University Press, Cambridge, 2006, xviii+512 pp.

\end{thebibliography}
\end{document}